\documentclass[a4paper,11pt]{article}
\usepackage{amsfonts}
\usepackage{amsmath}
\usepackage{amssymb}
\usepackage{amsthm}
\numberwithin{equation}{section}         
\usepackage{indentfirst}                
\usepackage{verbatim}                    
\usepackage{mathrsfs}
\usepackage{graphicx}

\theoremstyle{plain} \newtheorem{Lem}{Lemma}[section]       
\theoremstyle{plain}\newtheorem{theorem}{Theorem}[section]  
\theoremstyle{plain}    
\theoremstyle{remark} \newtheorem{Rem}{Remarks}[section]   
\theoremstyle{plain}\newtheorem{Pro}{Proposition}[section]  
\theoremstyle{plain}  
\allowdisplaybreaks   

\newcommand{\be}{\begin{equation}}
\newcommand{\ee}{\end{equation}}
\newcommand{\bpm}{\begin{pmatrix}}
\newcommand{\epm}{\end{pmatrix}}

\begin{document}


\begin{center}
\LARGE{Numerical Solutions of Jump Diffusions with Markovian Switching}\\
\end{center}
\begin{center}
Jun Ye, Kai Li\\

\end{center}
\begin{center}
\emph{Department of Mathematical Sciences, Tsinghua University,
Beijing 100084,
P.R.China\\
E-mail: jye@math.tsinghua.edu.cn}\\
[.5cm]
\end{center}

\def\abstractname{Abstract}
\begin{abstract}
In this paper we consider the numerical solutions for a class of
jump diffusions with Markovian switching. After briefly reviewing
necessary notions, a new jump-adapted efficient algorithm based on
the Euler scheme is constructed for approximating the exact
solution. Under some general conditions, it is proved that the
numerical solution through such scheme converge to the exact
solution. Moreover, the order of the error between the numerical
solution and the exact solution is also derived. Numerical
experiments are carried out to
show the computational efficiency of the approximation.
\textbf{MR(2000) Subject Classification} 65C20, 60H10\\
\end{abstract}

\section{Introduction}
Stochastic hybrid systems arise in numerous applications of systems
with frequent unpredictable structural changes, e.g. flexible
manufacturing systems, air traffic management, electric power
networks, mathematical finance and risk management etc. In typical
hybrid system, the state space consists of two parts: one component
takes values in Euclidean space while the other takes discrete
values. For various applications of stochastic hybrid systems, we
refer to \cite{3,5,4,2}.

One important classes of stochastic hybrid systems is the
following jump diffusion with Markovian switching:
\begin{equation}\label{eq:jm}
\begin{split}
y(t)=&y_0+\int_{0}^{t}b(y(s),r(s))ds+\int_{0}^{t}\sigma(y(s),r(s))dW(s)\\
&\quad+\int_{0}^{t}\int_{\Gamma}g(y(s^-),r(s),v)N(ds,dv),\\
\end{split}
\end{equation}
where $r(t)$ is a continuous-time Markov chain represents the
environment or modes. Such processes have been increasingly used in
modelling of the stochastic systems which are affected by randomly
occuring impulses as well as frequent regime switching usually
modelled by Markov chain. A typical application of (\ref{eq:jm})
stems from risk theory in insurance and finance where $y(t)$ is
regarded as a surplus process in Markov-modulated market.
Applications have also been reported in a wide range of fields such
as option pricing \cite{16,15} and flexible manufacturing systems
\cite{17}.

Since in many problems such L\'{e}vy stochastic differential
equations with Markovian switching cannot be solved explicitly, it
is important, from theoretical point of view, and even more for the
sake of various applications, to find their approximate solutions in
an explicit form or in a form suitable for application of some
numerical methods. Although the numerical solutions of stochastic
differential equations have been the focus of enormous research (see
e.g. \cite{25,26,6,7,8,9}), there has been significantly less work
on the numerical methods for diffusion processes or jump diffusion
processes with Markovian switching. In the literatures addressing
these problems, Yuan and Mao \cite{10} firstly considered the
numerical solutions of the following stochastic differential
equations with Markovian switching
\begin{equation}\label{eq:m}
\begin{split}
&dy(t)=f(y(t),r(t))dt+g(y(t),r(t))dW(t),\\
&P\{r(t+\delta)=j|r(t)=i\}=q_{ij}\delta+o(\delta)\quad i\neq j.
\end{split}
\end{equation}
It was proved that the numerical solutions of (\ref{eq:m}) under
Euler method converge to the exact solutions, and the order of error
was also estimated. Krystul and Bagchi \cite{11} extended
(\ref{eq:m}) to a more general switching diffusion processes with
state-dependent switching rates, and a approximation scheme for
first passage time of the processes was derived. Recently, Yin, Song
and Zhang \cite{19} proposed a numerical algorithm for
(\ref{eq:jm}), and proved that the algorithm converges to the
desired limit by means of a martingale problem formulation.

In this paper, we try to develop a jump-adapted numerical algorithm
for (\ref{eq:jm}), which is based on the Euler scheme. Different
from \cite{19}, the algorithm we proposed does not rely on a
constant-step-size scheme, instead, a series of jump-times driven by
a Poisson random measure are taken into account. Another difference
is that convergence properties in \cite{19} was proved in weak sense
while we prove it in strong sense. Meanwhile this paper focus on the
order of convergence, which is enlightened by the methods in
\cite{10}.

The rest of the paper is organized as follows. Section 2 briefly
reviews the existence and uniqueness of equation (\ref{eq:jm}) and
describe the modified Euler algorithm of numerical solutions.
Section 3 shows that the numerical solution under such algorithm
converge to the exact solution in $L^2$, and the order of the
error between the numerical and exact solution is derived. Section
4 gives two numerical examples to examine the performance of the
algorithm described in Section 3.

\section{Preliminary and algorithm}
Let $(\Omega,\mathcal{F},\{\mathcal{F}_t\}_{t\geq0},P)$ be a
complete probability space with a filtration
$\{\mathcal{F}_t\}_{t\geq0}$ satisfying the usual conditions.
Suppose that there is a finite set $S=\{1,2,...,N\}$, representing
the possible regimes of the environment. Let $\Gamma$ be a compact
subset of  $\mathbb{R}$ not including the origin, and
$b(\cdot,\cdot):\mathbb{R}^d \times S \rightarrow \mathbb{R}^d$,
$\sigma(\cdot,\cdot):\mathbb{R}^d \times S \rightarrow
\mathbb{R}^{d\times d}$, $g(\cdot,\cdot,\cdot):\mathbb{R}^d \times
S \times \Gamma\rightarrow \mathbb{R}^d$. Consider the following
jump diffusion with Markovian switching of the form
\begin{equation}\label{eq:jm1}
\begin{split}
y(t)=&y_0+\int_{0}^{t}b(y(s),r(s))ds+\int_{0}^{t}\sigma(y(s),r(s))dW(s)\\
&\quad+\int_{0}^{t}\int_{\Gamma}g(y(s^-),r(s),v)N(ds,dv),\\
\end{split}
\end{equation}
with initial value $y(0)=y_0\in\mathbb{R}^d$ and $r(0)=i_0\in S$,
where $W(t)=(W^{1}(t),\cdots,W^{d}(t))^T$ is a $d$-dimensional
$\mathcal{F}_t$-adapted standard Brownian motion, $N(dt,dv)$ is a
$\mathcal{F}_t$-Poisson point process on
$\mathbb{R}^{+}\times\Gamma$ with deterministic characteristic
measure $\Pi(dv)$ on a compact set $\Gamma\subset\mathbb{R}$, let
$\widetilde{N}(dt,dv)=N(dt,dv)-\Pi(dv)dt$ be the compensated Poisson
random measure on $\mathbb{R}^{+}\times\Gamma$, and $r(t)$ is a
continuous-time Markov chain taking value in a finite state space
$S$ with the generator $Q=(q_{ij})_{N\times N}$ given by
\begin{equation*}
P\{r(t+\delta)=j|r(t)=i\}=\begin{cases}
q_{ij}\delta+o(\delta), &\text{if $i\neq j$},\\[.5cm]
1+q_{ii}\delta+o(\delta), &\text{if $i=j$},
\end{cases}
\end{equation*}
provided $\delta\downarrow 0$, and
\[-q_{ii}=\sum_{i\neq j}q_{ij}<+\infty.\]

 We assume $W(t),N(t,\cdot)$ and $r(t)$ are independent. Denote
$N_t\triangleq N(t,\Gamma)$ by a homogeneous Poisson process with
intensity $\lambda=\Pi(\Gamma)$. It is well known that the
compensated measure of $N(dt,dv)$ is $\Pi(dv)dt$, hence the centered
Poisson measure is $\widetilde{N}(dt,dv)$ which is a martingale on
$(\Omega,\mathcal{F},\{\mathcal{F}_t\}_{t\geq0},P)$ (see \cite{23}).
\begin{Rem}
when $d=1$, we assume that the characteristic measure
$\Pi(dv)=\lambda F(dv)$ where $F(dv)$ is a probability distribution
on $\Gamma\subseteq\mathbb{R}^{+}$ and $\lambda$ is a positive
constant. In this case, the integral with respect to the random
measure $N(dt,dv)$ is simply a compound Poisson process. We have
$\int_0^t\int_{\Gamma}vN(dt,dv)=\sum_{i=1}^{N_t}\xi_i$, where
$\{N_t,t\geq 0\}$ is a process with intensity $\lambda$ and
$\{\xi_i,i\in\mathbb{N}\}$ is a sequence of i.i.d. random variables
with common distribution $F$.
\end{Rem}

\subsection{Existence and uniqueness}
Under certain conditions we can establish the existence of a
pathwise unique solution of (\ref{eq:jm1}). Here we make the following global
Lipschitz and linear growth assumptions:\\[.2cm]
($\mathcal{H}$1) For all $\ (i,x,y,v)\in
S\times\mathbb{R}^d\times\mathbb{R}^d\times\Gamma$, there exists a
constant $L_1>0$ such that
\[|b(x,i)-b(y,i)|^2+|\sigma(x,i)-\sigma(y,i)|^2+\int_{\Gamma}
|g(x,i,v)-g(y,i,v)|^2\Pi(dv)\leq L_1|x-y|^2.\]\\
($\mathcal{H}$2) For all $\ (i,x,v)\in
S\times\mathbb{R}^d\times\Gamma$, there exists a constant $L_2>0$
such that
\[\int_{\Gamma}
|g(x,i,v)|^2\Pi(dv)\leq L_2(1+|x|^2).\]\\
($\mathcal{H}$3) Let $K$ be the support of $g(\cdot,\cdot,\cdot)$
and $U$ be the projection of $K$ on $\Gamma\subset\mathbb{R}$, then
assume that $\Pi(U)<\infty$.

 We denote
by $|\cdot|$ the Euclidean norm for vectors or the trace
 norm for matrices.

\begin{Rem}It's easy to show that if $b(\cdot,\cdot),\
\sigma(\cdot,\cdot)$ satisfy global Lipschitz condition, then they
also satisfy linear growth condition, i.e. for all
$(x,i)\in\mathbb{R}^d\times S$, there exists a constant $L_3$ such
that
\[|b(x,i)|^2+|\sigma(x,i)|^2\leq L_3(1+|x|^2).\]
To show this, just let
$L_3=2L_1\vee\max(2|b(0,i)|^2+2|\sigma(0,i)|^2:i\in S)$. For
convenience, we will set $L=L_1\vee L_2\vee L_3$ in the following
of this paper.
\end{Rem}

\begin{theorem}
If $b(x,i)$, $\sigma(x,i)$ and $g(x,i,v)$ satisfy the conditions
($\mathcal{H}$1), ($\mathcal{H}$2), ($\mathcal{H}$3), and suppose
$W(t),r(t),N(dt,dv)$ be independent. Then there exists a unique
d-dimensional $\mathcal{F}_t$-adapted right-continuous process
$y(t)$ with left-hand limits which satisfies equation (2.1) and
such that $y(0)=y_0$ and $r(0)=i_0$ a.s.
\end{theorem}

\begin{proof}
See \cite{13,12}。
\end{proof}

\subsection{Algorithm} Now we turn our attention to numerical
algorithm. Given $\Delta>0$ as a step size, denote
$\{t'_i\}_{i\geq1}$ the usual equidistant time discretization of a
bounded interval $[0,T]$, i.e. $t'_0=0$,
 $t'_i-t'_{i-1}=\Delta$, if $t'_{n-1}<T\leq t'_n$ then set $t'_n=T$.
Suppose $\tau'_j=inf\{t\geq 0,N_t\geq j\}, j=1,2,\cdots$ are the
jump times of $N_t\triangleq N(t,\Gamma)$. Apparently,
$\tau_j=\tau'_j\wedge T,j=1,2,\cdots$ are also stopping times. Now
we take a new time discretization
$\{t_k\}_{k\geq1}=\{t'_i\}_{i\geq1}\cup\{\tau_j\}_{j\geq1}$
satisfying $t_{k+1}-t_{k}\leq\Delta$ for all $k\geq1$. This type of
grid using the jump-times as additional grid points is usually
refered to as a jump-adapted method, see for example \cite{25}.

Since $N(dt,dv)$ and $r(t)$ are independent,  for given partition
$\{t_k\}_{k\geq 1}$, $\{r(t_k),k=1,2,\cdots\}$ is a discrete
Markov chain with transition probability matrix $(P(i,j))_{N\times
N}$, here $P(i,j)= P(r(t_{k+1})=j|r(t_{k})=i)$ is the $ij$th entry
of the matrix $e^{(t_{k+1}-t_{k})Q}$, thus we could use following
recursion procedure to simulate the discrete Markov chain
$\{r(t_k)\}_{k\geq1}$, suppose $r(t_k)=i_1$ and generate a random
number $\xi$ which is uniformly distributed in $[0,1]$, then we
define
\begin{equation*} r(t_{k+1})=\begin{cases}
i_2, &\text{if $\ i_2\in S-\{N\}$ and $\ \sum^{i_2-1}_{j=1}P(i_1,j)\leq\xi<\sum^{i_2}_{j=1}P(i_1,j)$},\\[.5cm]
N, &\text{if $\ \sum^{N-1}_{j=1}P(i_1,j)\leq\xi.$}
\end{cases}
\end{equation*}
Repeating this procedure a trajectory of $\{r(t_k)\}_{k\geq1}$ can
be simulated.

Now we could define the Euler approximation solution of (\ref{eq:jm1}). Let
\[X_{t_0}=y_0,\ r_{t_0}=i_0,\]
\begin{equation}\label{eq:em}
\begin{split}
X_{t_{k+1}}&=X_{t_k}+b(X_{t_k},r_{t_k})(t_{k+1}-t_{k})+\sigma(X_{t_k},r_{t_k})(W({t_{k+1}})-W({t_k}))\\
&\quad+\sum_{j=1}^{N_T}g(X_{t_k},r_{t_k},v_j)I_{\{t_{k+1}=\tau_j\}}.
\end{split}
\end{equation}
Let
\[\bar{X}(t)=X_{t_k},\ \bar{r}(t)=r_{t_k}\         (t_k\leq t<t_{k+1}).\]

Hence we can define the approximation solution $X(t)$ on the
entire interval $[0,T]$ by
\begin{equation}\label{eq:emc}
\begin{split}
X(t)&=y_0+\int_{0}^{t}b(\bar{X}(s),\bar{r}(s))ds+\int_{0}^{t}\sigma(\bar{X}(s),\bar{r}(s))dW(s)\\
&\quad+\int_{0}^{t}\int_{\Gamma}
g(\bar{X}(s^-),\bar{r}(s),v)N(ds,dv).
\end{split}
\end{equation}

\begin{Rem}
 $\sum_{j=1}^{N_T}g(X_{t_k},r_{t_k},v_j)I_{\{t_{k+1}=\tau_j\}}$
only take nonzero values on stopping times $\{\tau_j\}_{j\geq1}$,
thus $X(t)$ is continuous on $[0,T]\setminus\{\tau_j\}_{j\geq1}$.
Moreover, we have $X(t_k)=\bar{X}(t_k)=X_{t_k}$.
\end{Rem}

\section{Convergence with the Lipschitz and linear growth conditions}
In this section, we will prove that the numerical solution $X(t)$
converges to the exact solution $y(t)$ in $L^2$ as step size
$\Delta\downarrow0$, and the order of convergence is one-half,
i.e.
\begin{equation}
E(\sup_{0\leq t\leq T}|X(t)-y(t)|^2)\leq C\Delta+o(\Delta).
\end{equation}

To begin with, we need the following lemma.
\begin{Lem}
Under conditions $(\mathcal{H}1)$, $(\mathcal{H}2)$ and
$(\mathcal{H}3)$, there exists a constant $M$ which is dependent on
$T$, $L$, $y_0$, but independent of $\Delta$, such that
\begin{equation}\label{eq:bound}
E(\sup_{0\leq t\leq T}|y(t)|^2)\vee E(\sup_{0\leq t\leq
T}|X(t)|^2)\leq M.
\end{equation}
\end{Lem}
\begin{proof} From  H\"{o}lder inequality, we have
\begin{equation*}
\begin{split}
|y(t)|^2 &=
|y_0+\int_{0}^{t}b(y(s),r(s))ds+\int_{0}^{t}\sigma(y(s),r(s))dW(s)\\
&\quad+\int_{0}^{t}\int_{\Gamma}g(y(s^-),r(s),v)N(ds,dv)|^2 \\
&\leq4|y_0|^2+4|\int_{0}^{t}b(y(s),r(s))ds|^2+4|\int_{0}^{t}\sigma(y(s),r(s))dW(s)|^2\\
&\quad+4|\int_{0}^{t}\int_{\Gamma}g(y(s^-),r(s),v)N(ds,dv)|^2\\
&\leq4|y_0|^2+4t\int_{0}^{t}|b(y(s),r(s))|^2ds+4|\int_{0}^{t}\sigma(y(s),r(s))dW(s)|^2\\
&\quad+4|\int_{0}^{t}\int_{\Gamma}g(y(s^-),r(s),v)N(ds,dv)|^2.
\end{split}
\end{equation*}
Thus for any $0\leq T'\leq T$, we have
\begin{equation}\label{eq:lemp1}
\begin{split}
 \sup_{0\leq t\leq T'}|y(t)|^2 &\leq
4|y_0|^2+4T\int_{0}^{T'}|b(y(s),r(s))|^2ds \\
&\quad+4\sup_{0\leq t\leq T'}|\int_{0}^{t}\sigma(y(s),r(s))dW(s)|^2 \\
&\quad+4\sup_{0\leq t\leq T'}|\int_{0}^{t}\int_{\Gamma}
g(y(s^-),r(s),v)N(ds,dv)|^2.
\end{split}
\end{equation}
From $(\mathcal{H}1)$ and Remark 2.2, we have
\begin{equation}\label{eq:lemp2}
\begin{split}
E\int_{0}^{T'}|b(y(s),r(s))|^2ds &\leq E\int_{0}^{T'}L(1+|y(s)|^2)ds\\
&\leq C_1+C_1\int_{0}^{T'}E|y(s)|^2ds.
\end{split}
\end{equation}
Here $C_1$ is a positive constant, in fact, $C_1=T'L\vee L$ in
(\ref{eq:lemp2}). For convenience, in the following of this paper, we may
frequently denote the related constants by $C_k$. In that case, we
only mean there exist such a positive constant rather than some
specific one.

 Since
$\displaystyle\int_{0}^{t}\sigma(y(s),r(s))dW(s)$ is martingale
(see \cite{A1}), therefore, from Doob martingale inequality and
($\mathcal{H}$1), we have

\begin{equation}\label{eq:lemp3}
\begin{split}
&E(\sup_{0\leq t\leq
T'}|\int_{0}^{t}\sigma(y(s),r(s))dW(s)|^2)\\
&\quad\leq
4E|\int_{0}^{T'}\sigma(y(s),r(s))dW(s)|^2\\
&\quad=4E\int_{0}^{T'}|\sigma(y(s),r(s))|^2ds\\
&\quad\leq 4E\int_{0}^{T'}L(1+|y(s)|^2)ds\\
&\quad\leq C_2+C_2\int_{0}^{T'}E|y(s)|^2ds.
\end{split}
\end{equation}

For the last term of (\ref{eq:lemp1}), notice that
$\widetilde{N}(dt,dv)=N(dt,dv)-\Pi(dv)dt$ is the compensated
Poisson random measure, and $\displaystyle
\int_{0}^{t}\int_{\Gamma}g(y(s^-),r(s),v)\widetilde{N}(ds,dv)$ is
martingale (see \cite{23}), we have
\begin{equation}\label{eq:lemp4}
\begin{split}
&E(\sup_{0\leq t\leq T'}|\int_{0}^{t}\int_{\Gamma}
g(y(s^-),r(s),v)N(ds,dv)|^2)\\
&\quad=E(\sup_{0\leq t\leq T'}|\int_{0}^{t}\int_{\Gamma}
g(y(s^-),r(s),v)\widetilde{N}(ds,dv)\\
&\quad\quad+\int_{0}^{t}\int_{\Gamma}
g(y(s^-),r(s),v)\Pi(dv)ds|^2)\\
&\quad\leq 2E(\sup_{0\leq t\leq T'}|\int_{0}^{t}\int_{\Gamma}
g(y(s^-),r(s),v)\widetilde{N}(ds,dv)|^2)\\
&\quad\quad+2 E(\sup_{0\leq t\leq T'}|\int_{0}^{t}\int_{\Gamma}
g(y(s^-),r(s),v)\Pi(dv)ds|^2).\\
\end{split}
\end{equation}
Hence, by Doob martingale inequality, we have
\begin{equation}\label{eq:lemp5}
\begin{split}
&E(\sup_{0\leq t\leq T'}|\int_{0}^{t}\int_{\Gamma}
g(y(s^-),r(s),v)\widetilde{N}(ds,dv)|^2)\\
&\quad\leq4E|\int_{0}^{T'}\hspace{-5pt}\int_{\Gamma}
g(y(s^-),r(s),v)\widetilde{N}(ds,dv)|^2\\
&\quad =4E\int_{0}^{T'}\hspace{-5pt}\int_{\Gamma}
|g(y(s^-),r(s),v)|^2\Pi(dv)ds.
\end{split}
\end{equation}
On the other hand, we can get
\begin{equation}\label{eq:lemp7}
\begin{split}
&E(\sup_{0\leq t\leq T'}|\int_{0}^{t}\int_{\Gamma}
g(y(s^-),r(s),v)\Pi(dv)ds|^2)\\
&\quad\leq E(\sup_{0\leq t\leq
T'}t\int_{0}^{t}\int_{\Gamma} |g(y(s^-),r(s),v)|^2\Pi(dv)ds)\\
&\quad\leq TE\int_{0}^{T'}\hspace{-5pt}\int_{\Gamma}
|g(y(s^-),r(s),v)|^2\Pi(dv)ds.
\end{split}
\end{equation}
Combining (\ref{eq:lemp4}), (\ref{eq:lemp5}), (\ref{eq:lemp7}) and
($\mathcal{H}$2), we obtain that
\begin{equation}\label{eq:lemp8}
\begin{split}
&E(\sup_{0\leq t\leq T'}|\int_{0}^{t}\int_{\Gamma}
g(y(s^-),r(s),v)N(ds,dv)|^2)\\
&\quad\leq (8+2T)E\int_{0}^{T'}\hspace{-5pt}\int_{\Gamma}
|g(y(s^-),r(s),v)|^2\Pi(dv)ds\\
&\quad\leq (8+2T)E\int_{0}^{T'}L(1+|y(s^-)|)^2ds\\
&\quad\leq C_3+C_3\int_{0}^{T'}E|y(s^-)|^2ds.
\end{split}
\end{equation}
Take expectation of (\ref{eq:lemp1}) and substitute(\ref{eq:lemp2}), (\ref{eq:lemp3}), (\ref{eq:lemp8}) into it,
we have
\begin{equation}\label{eq:lemp9}
\begin{split}
&E(\sup_{0\leq t\leq T'}|y(t)|^2)\\
&\quad\leq4|y_0|^2+(C_1+C_2)\int_{0}^{T'}E|y(t)|^2dt+C_3\int_{0}^{T'}E|y(t^-)|^2dt\\
&\quad\leq C_4+C_4\int_{0}^{T'}E(\sup_{0\leq s\leq t}|y(s)|^2)dt.
\end{split}
\end{equation}
From Gronwall inequality, we can show that
\[E(\sup_{0\leq t\leq T}|y(t)|^2)\leq C_4(1+C_4Te^{C_4T})=M_1.\]
Similar to above procedure, we know that for $X(t)$
\[E(\sup_{0\leq t\leq T'}|X(t)|^2)\leq
4|y_0|^2+(C_1+C_2)\int_{0}^{T'}E|\bar{X}(t)|^2dt+C_3\int_{0}^{T'}E|\bar{X}(t^-)|^2dt.\]
According to Remark 2.3
\[\bar{X}(t)=\bar{X}(t^-)=X_{t_k}=X(t_k),\quad t_k<t<t_{k+1},\]
\[\bar{X}(t)=X_{t_k}=X(t_k),\    \bar{X}(t^-)=X_{t_{k-1}}=X(t_{k-1}), \quad t=t_k.\]
In both cases we have
\[|\bar{X}(t)|^2\vee|\bar{X}(t^-)|^2\leq \sup_{0\leq s\leq t}|X(s)|^2.\]
Thus by Gronwall inequality
\[E(\sup_{0\leq t\leq T}|X(t)|^2)\leq M_2.\]
Finally there exists a positive constant $M$ such that
\[E(\sup_{0\leq t\leq T}|y(t)|^2)\vee
E(\sup_{0\leq t\leq T}|X(t)|^2)\leq M.\]
\end{proof}

Now we come back to our main theorem, the proof below mainly refers
to Yuan and Mao \cite{10}.
\begin{theorem}
Assume the jump diffusion with Markovian switching (\ref{eq:jm1}) defined on $(\Omega,\mathcal{F},\{\mathcal{F}_t\}_{t\geq0},P)$ satisfying\\[.5cm]
\textbf{a}) $W(t),r(t),N(dt,dv)$ are independent,\\[.5cm]
\textbf{b}) $b(\cdot,\cdot),\ \sigma(\cdot,\cdot),\
g(\cdot,\cdot,\cdot)$ satisfy conditions $(\mathcal{H}1)$, $(\mathcal{H}2)$ and $(\mathcal{H}3)$, \\[.5cm]
then the unique strong solution $y(t)$ and numerical solution
$X(t)$ obtained in section 2.2 satisfying:
\[E(\sup_{0\leq t\leq T}|X(t)-y(t)|^2)\leq C\Delta+o(\Delta).\]
where $C$ is a positive constant independent of $\Delta$.
\end{theorem}

\begin{proof}
For every $0\leq T'\leq T$, similar to Lemma 3.1 we can easily
verify that
\begin{equation}\label{eq:theop1}
\begin{split}
&E(\sup_{0\leq t\leq T'}|X(t)-y(t)|^2)\\
&\quad\leq 3TE\int_{0}^{T'}|b(\bar{X}(s),\bar{r}(s))-b(y(s),r(s))|^2ds\\
&\quad\quad+12E\int_{0}^{T'}|\sigma(\bar{X}(s),\bar{r}(s))-\sigma(y(s),r(s))|^2ds \\
&\quad\quad+CE\int_{0}^{T'}\hspace{-5pt}\int_{\Gamma}|g(\bar{X}(s^-),\bar{r}(s),v)-g(y(s^-),r(s),v)|^2\Pi(dv)ds.
\end{split}
\end{equation}
We focus on the third part of the right side
\begin{equation}\label{eq:theop2}
\begin{split}
&E\int_{0}^{T'}\hspace{-5pt}\int_{\Gamma}|g(\bar{X}(s^-),\bar{r}(s),v)-g(y(s^-),r(s),v)|^2\Pi(dv)ds\\
&\quad\leq2E\int_{0}^{T'}\hspace{-5pt}\int_{\Gamma}|g(\bar{X}(s^-),r(s),v)-g(y(s^-),r(s),v)|^2\Pi(dv)ds\\
&\quad\quad+2E\int_{0}^{T'}\hspace{-5pt}\int_{\Gamma}|g(\bar{X}(s^-),\bar{r}(s),v)-g(\bar{X}(s^-),r(s),v)|^2\Pi(dv)ds.\\
\end{split}
\end{equation}
By condition ($\mathcal{H}$1), we have
\begin{equation}\label{eq:theop3}
\begin{split}
&E\int_{0}^{T'}\hspace{-5pt}\int_{\Gamma}|g(\bar{X}(s^-),r(s),v)-g(y(s^-),r(s),v)|^2\Pi(dv)ds\\
&\quad \leq2L^2\int_{0}^{T'}E|\bar{X}(s^-)-y(s^-)|^2ds.
\end{split}
\end{equation}
By ($\mathcal{H}$2), we can compute that
\begin{equation}\label{eq:theop4}
\begin{split}
&E\int_{0}^{T}\int_{\Gamma}|g(\bar{X}(s^-),\bar{r}(s),v)-g(\bar{X}(s^-),r(s),v)|^2\Pi(dv)ds\\
&\quad=E[E[\int_{0}^{T}\int_{\Gamma}|g(\bar{X}(s^-),\bar{r}(s),v)\\
&\quad\quad-g(\bar{X}(s^-),r(s),v)|^2\Pi(dv)ds\mid N_T,\tau_1,...\tau_{N_T}]]\\
&\quad=E[\sum_{k\geq0}
E[\int_{t_k}^{t_{k+1}}\int_{\Gamma}|g(\bar{X}(s^-),\bar{r}(s),v)\\
&\quad\quad-g(\bar{X}(s^-),r(s),v)|^2\Pi(dv)ds\mid N_T,\tau_1,...\tau_{N_T}]]\\
&\quad=E[\sum_{k\geq0}E[\int_{t_k}^{t_{k+1}}\int_{\Gamma}|g(\bar{X}(t_k),r(t_k),v)-g(\bar{X}(t_k),r(s),v)|^2\\
&\quad\quad \times I_{\{r(s)\neq r(t_k)\}}\Pi(dv)ds\mid N_T,\tau_1,...\tau_{N_T}]]\\
&\quad\leq2E[\sum_{k\geq0}E[\int_{t_k}^{t_{k+1}}\int_{\Gamma}(|g(\bar{X}(t_k),r(t_k),v)|^2+|g(\bar{X}(t_k),r(s),v)|^2)\\
&\quad\quad \times I_{\{r(s)\neq r(t_k)\}}\Pi(dv)ds\mid N_T,\tau_1,...\tau_{N_T}]]\\
&\quad\leq CE[\sum_{k\geq0}
E[\int_{t_k}^{t_{k+1}}[1+|\bar{X}(t_k)|^2]I_{\{r(s)\neq r(t_k)\}}ds\mid N_T,\tau_1,...\tau_{N_T}]].\\
\end{split}
\end{equation}
Note that given $\{N_T,\tau_1,...\tau_{N_T}\}$, $1+|\bar{X}(t_k)|^2$
and $I_{\{r(s)\neq r(t_k)\}}$ are conditionally independent with
respect to $r(t_k)$, thus
\begin{equation}\label{eq:theop5}
\begin{split}
&E\int_{t_k}^{t_{k+1}}[1+|\bar{X}(t_k)|^2]I_{\{r(s)\neq r(t_k)\}}ds\\
&\quad=\int_{t_k}^{t_{k+1}}E[E[(1+|\bar{X}(t_k)|^2)I_{\{r(s)\neq r(t_k)\}}\mid r(t_k)]]ds\\
&\quad=\int_{t_k}^{t_{k+1}}E[E[(1+|\bar{X}(t_k)|^2)\mid
r(t_k)]E[I_{\{r(s)\neq r(t_k)\}}\mid r(t_k)]]ds.
\end{split}
\end{equation}
On the other hand, by the Markov property, we get
\begin{equation}\label{eq:theop6}
\begin{split}
E[I_{\{r(s)\neq r(t_k)\}}\mid r(t_k)]]&=\sum_{i\in S}I_{\{r(t_k)=i\}}P(r(s)\neq i\mid r(t_k)=i)\\
&=\sum_{i\in S}I_{\{r(t_k)=i\}}(-q_{ii}(s-t_k)+o(s-t_k))\\
&\leq \sum_{i\in S}I_{\{r(t_k)=i\}}(max(-q_{ii})\Delta+o(\Delta))\\
&\leq C\Delta+o(\Delta).
\end{split}
\end{equation}
Substituting (\ref{eq:theop6}) into (\ref{eq:theop5}) gives
\begin{equation*}\label{eq:barwq}
\begin{split}
&E\int_{t_k}^{t_{k+1}}[1+|\bar{X}(t_k)|^2]I_{\{r(s)\neq r(t_k)\}}ds\\
&\quad\leq \int_{t_k}^{t_{k+1}}E[E[(1+|\bar{X}(t_k)|^2)\mid r(t_k)](C\Delta+o(\Delta))]ds\\
&\quad\leq (C\Delta+o(\Delta))\int_{t_k}^{t_{k+1}}E[E[(1+|\bar{X}(t_k)|^2)\mid r(t_k)]]ds\\
&\quad=(C\Delta+o(\Delta))\int_{t_k}^{t_{k+1}}E[1+|\bar{X}(t_k)|^2]ds.
\end{split}
\end{equation*}
By Lemma 3.1 there exists a positive constant $M$ such that
\[E[1+|\bar{X}(t_k)|^2]\leq 1+E(\sup_{0\leq t\leq T}|X(t)|^2)\leq M,\]
then
\begin{equation}\label{eq:theop7}
E\int_{t_k}^{t_{k+1}}[1+|\bar{X}(t_k)|^2]I_{\{r(s)\neq r(t_k)\}}ds
\leq (t_{k+1}-t_k)(C\Delta+o(\Delta)).
\end{equation}
Substitute (\ref{eq:theop7}) into (\ref{eq:theop4})
\begin{equation}\label{eq:theop8}
\begin{split}
&E\int_{0}^{T}\int_{\Gamma}|g(\bar{X}(s^-),\bar{r}(s),v)-g(\bar{X}(s^-),r(s),v)|^2\Pi(dv)ds\\
&\quad\leq CE[\sum_{k\geq0}E[(t_{k+1}-t_k)(C\Delta+o(\Delta))\mid N_T,\tau_1,...\tau_{N_T}]]\\
&\quad=CE[\sum_{k\geq0}(t_{k+1}-t_k)(C\Delta+o(\Delta))]\\
&\quad\leq C\Delta+o(\Delta).
\end{split}
\end{equation}
Here we use the fact that $(t_{k+1}-t_k)(C\Delta+o(\Delta))$ are
measurable with respect to the $\sigma$-algebra generated by
$\{N_T,\tau_1,...\tau_{N_T}\}$.\\
Now substitute (\ref{eq:theop3}), (\ref{eq:theop8}) into (\ref{eq:theop2}), we obtain that
\begin{equation}\label{eq:theop9}
\begin{split}
&E\int_{0}^{T'}\hspace{-5pt}\int_{\Gamma}|g(\bar{X}(s^-),\bar{r}(s),v)-g(y(s^-),r(s),v)|^2\Pi(dv)ds\\
&\quad\leq2L^2\int_{0}^{T'}E|\bar{X}(s^-)-y(s^-)|^2ds+C\Delta+o(\Delta).
\end{split}
\end{equation}
Similar to (\ref{eq:theop9}), we have
\begin{equation}\label{eq:theop10}
\begin{split}
&E\int_{0}^{T'}|b(\bar{X}(s),\bar{r}(s))-b(y(s),r(s))|^2ds\\
&\quad\leq2L^2\int_{0}^{T'}E|\bar{X}(s)-y(s)|^2ds+C\Delta+o(\Delta),
\end{split}
\end{equation}
\begin{equation}\label{eq:theop11}
\begin{split}
&E\int_{0}^{T'}|\sigma(\bar{X}(s),\bar{r}(s))-\sigma(y(s),r(s))|^2ds \\
&\quad\leq2L^2\int_{0}^{T'}E|\bar{X}(s)-y(s)|^2ds+C\Delta+o(\Delta).
\end{split}
\end{equation}
Substituting (\ref{eq:theop9}), (\ref{eq:theop10}), (\ref{eq:theop11}) into (\ref{eq:theop1}) shows that
\begin{equation}\label{eq:theop12}
\begin{split}
&E(\sup_{0\leq t\leq T'}|X(t)-y(t)|^2)\\
&\quad\leq C\int_{0}^{T'}E|\bar{X}(s)-y(s)|^2ds\\
&\quad\quad+C\int_{0}^{T'}E|\bar{X}(s^-)-y(s^-)|^2ds+C\Delta+o(\Delta).
\end{split}
\end{equation}
Note that
\begin{equation}\label{eq:theop13}
E|\bar{X}(s)-y(s)|^2\leq 2E|X(s)-y(s)|^2+2E|\bar{X}(s)-X(s)|^2,
\end{equation}
\begin{equation}\label{eq:theop14}
E|\bar{X}(s^-)-y(s^-)|^2\leq2E|X(s^-)-y(s^-)|^2+2E|\bar{X}(s^-)-X(s^-)|^2.
\end{equation}
Suppose $t_k\leq s< t_{k+1}$, then
\begin{align*}
|X(s)-\bar{X}(s)|^2&=|X(s)-X(t_k)|^2\\[.2cm]
&=|\int_{t_k}^{s}b(\bar{X}(u),\bar{r}(u))du+\int_{t_k}^{s} \sigma(\bar{X}(u),\bar{r}(u))dW(u)\\[.3cm]
&\quad+\int_{t_k}^{s}\int_{\Gamma}g(\bar{X}(u^-),\bar{r}(u),v)N(du,dv)|^2\\[.3cm]
&=|b(X_{t_k},r_{t_k})(s-t_k)+\sigma(X_{t_k},r_{t_k})(W(s)-W({t_k}))\\[.3cm]
&\quad+\int_{t_k}^{s}\int_{\Gamma}g(X_{t_k},r_{t_k},v)N(du,dv)|^2\\[.3cm]
&\leq3|b(X_{t_k},r_{t_k})|^2(s-t_k)^2+3|\sigma(X_{t_k},r_{t_k})|^2(W(s)-W({t_k}))^2\\[.3cm]
&\quad+3|\int_{t_k}^{s}\int_{\Gamma}g(X_{t_k},r_{t_k},v)N(du,dv)|^2.\\[.3cm]
\end{align*}
Since for given $\{t_k\}_{k\geq1}$, $W(s)-W({t_k})$ and
$(X({t_k}),r_{t_k})$ are independent, we have
\begin{align*}
&E|X(s)-\bar{X}(s)|^2\\
&\quad\leq 3E[|b(X_{t_k},r_{t_k})|^2(s-t_k)^2+|\sigma(X_{t_k},r_{t_k})|^2(W(s)-W({t_k}))^2\\[.3cm]
&\quad\quad+|\int_{t_k}^{s}\int_{\Gamma}g(X_{t_k},r_{t_k},v)N(du,dv)|^2]\\[.3cm]
&\quad=3E[E[|b(X_{t_k},r_{t_k})|^2(s-t_k)^2+|\sigma(X_{t_k},r_{t_k})|^2(W(s)-W({t_k}))^2\\[.3cm]
&\quad\quad+|\int_{t_k}^{s}\int_{\Gamma}g(X_{t_k},r_{t_k},v)N(du,dv)|^2\mid t_k]]\\[.3cm]
&\quad\leq 3E[E[|b(X_{t_k},r_{t_k})|^2\Delta^2\mid t_k]+E[|\sigma(X_{t_k},r_{t_k})|^2(W(s)-W({t_k}))^2\mid t_k]\\[.3cm]
&\quad\quad+E[|\int_{t_k}^{s}\int_{\Gamma}g(X_{t_k},r_{t_k},v)N(du,dv)|^2\mid t_k]]\\[.3cm]
&\quad=3\Delta^2E[|b(X_{t_k},r_{t_k})|^2]+3E[E[|\sigma(X_{t_k},r_{t_k})|^2\mid t_k]E[(W(s)-W({t_k}))^2\mid t_k]]\\[.3cm]
&\quad\quad+3E[E[\int_{t_k}^{s}\int_{\Gamma}|g(X_{t_k},r_{t_k},v)|^2\Pi(dv)du\mid t_k]]\\[.3cm]
&\quad=3\Delta^2E[|b(X_{t_k},r_{t_k})|^2]+3E[E[|\sigma(X_{t_k},r_{t_k})|^2\mid t_k](s-t_k)]\\[.3cm]
&\quad\quad+3E[E[\int_{\Gamma}|g(X_{t_k},r_{t_k},v)|^2\Pi(dv)\mid t_k](s-t_k)]\\[.3cm]
&\quad\leq 3E[|b(X_{t_k},r_{t_k})|^2]+3E[E[|\sigma(X_{t_k},r_{t_k})|^2\mid t_k]\Delta]\\[.3cm]
&\quad\quad+3E[E[\int_{\Gamma}|g(X_{t_k},r_{t_k},v)|^2\Pi(dv)\mid t_k]\Delta]\\[.3cm]
&\quad=3\Delta^2E[|b(X_{t_k},r_{t_k})|^2]+3\Delta E[|\sigma(X_{t_k},r_{t_k})|^2]\\[.3cm]
&\quad\quad+3\Delta E[\int_{\Gamma}|g(X_{t_k},r_{t_k},v)|^2\Pi(dv)].
\end{align*}
Hence by ($\mathcal{H}$1) and Lemma 3.1, we get
\begin{equation}\label{eq:theop15}
E|\bar{X}(s)-X(s)|^2\leq C\Delta+o(\Delta).
\end{equation}
Similarly, we can show that
\begin{equation}\label{eq:theop16}
E|\bar{X}(s^-)-X(s^-)|^2\leq C\Delta+o(\Delta).\\[.3cm]
\end{equation}

Substituting (\ref{eq:theop13}), (\ref{eq:theop14}), (\ref{eq:theop15}), (\ref{eq:theop16})
into (\ref{eq:theop12}) immediately
shows that
\begin{equation*}
\begin{split}
&E(\sup_{0\leq t\leq T'}|X(t)-y(t)|^2)\\
&\quad\leq C\int_{0}^{T'}(E|\bar{X}(s)-y(s)|^2+E|\bar{X}(s^-)-y(s^-)|^2)ds+C\Delta+o(\Delta)\\
&\quad\leq C\int_{0}^{T'}E(\sup_{0\leq s\leq
t}|X(s)-y(s)|^2)dt+C\Delta+o(\Delta).
\end{split}
\end{equation*}
Therefore, from Gronwall inequality we obtain that
\[E(\sup_{0\leq t\leq T}|X(t)-y(t)|^2)\leq C\Delta+o(\Delta).\]
The proof is complete.
\end{proof}

\section{Numerical examples}
In this section, we discuss two numerical examples to illustrate
our theory established in the previous sections. It is shown that
the computer simulation based on our numerical method is feasible
and efficient.\\

\noindent \textsl{Example 4.1}. Consider the following
one-dimensional geometric L\'{e}vy processes
\begin{equation}\label{eq:gl}
\begin{split}
y(t)&=y_0+\int_{0}^{t}y(s)\mu(r(s))ds+\int_{0}^{t}y(s)\sigma(r(s))dW(s)\\
&\quad+\int_{0}^{t}y(s^-)g(r(s))dN(s),
\end{split}
\end{equation}
with initial value $y(0)=y_0$, $r(0)=i_0$, where $W(t)$ is a
scalar Brownian motion, $N(t)$ is a Poisson process with intensity
$\lambda>0$ and $r(t)$ is a right-continuous Markov chain taking
values in $S$, suppose that $r(t)$, $W(t)$, $N(t)$ are
independent. It is obvious that the assumptions ($\mathcal{H}$1),
($\mathcal{H}$2), ($\mathcal{H}$3) are satisfied for equation
(\ref{eq:gl}). It is well known that equation (\ref{eq:gl}) has an explicit
solution
\begin{equation}\label{eq:gls}
\begin{split}
y(t)=&y_0\textrm{exp}[\int_{0}^{t}\mu(r(s))ds+\int_{0}^{t}\sigma(r(s))dW(s)-\frac{1}{2}\int_{0}^{t}\sigma^2(r(s))ds]\\
&\times\prod_{n=1}^{N(t)}(1+g(r(\tau_n))),
\end{split}
\end{equation}
where $\tau_n$ are jump times determined by Poisson process
$N(t)$.

 As described in Section 2, we get the numerical solution of (\ref{eq:gl})
\begin{equation}\label{eq:gln}
\begin{split}
X(t)&=y_0+\int_{0}^{t}\bar{X}(s)\mu(\bar{r}(s))ds+\int_{0}^{t}\bar{X}(s)\sigma(\bar{r}(s))dW(s)\\
&\quad+\int_{0}^{t}\bar{X}(s)g(\bar{r}(s))dN(s).
\end{split}
\end{equation}

 By virtue of Theorem 3.1, it follows that the solutions $y(t)$ and $X(t)$ are close in the
$L^2$-norm, with the rate less than $C\Delta$ when $\Delta\downarrow
0$. To examine this convergence result, we firstly simulate the
trajectory of $y(t)$ and $X(t)$ for sufficient small stepsize
$\Delta>0$, then use simulated value
$\widehat{E}(\displaystyle\sup_{t_k\in [0,T]}|X(t_k)-y(t_k)|^2)$
 at new division points $t_k$ to estimate $E(\displaystyle\sup_{0\leq t\leq T}|X(t)-y(t)|^2)$,
where $\widehat{E}$ denotes sample mean.

 For simulation reason, it is convenient to transform (\ref{eq:gls})
 into following recursion form with $y(t_0)=y_0$,
\begin{equation}\label{eq:glsn}
\begin{split}
y(t_{k+1})&
\approx y(t_k)\textrm{exp}[(t_{k+1}-t_k)\mu(r_{t_k})+(W({t_{k+1}})-W({t_k}))\sigma(r_{t_k})\\
&\quad
-\frac{1}{2}(t_{k+1}-t_k)\sigma^2(r_{t_k})]\times\sum_{j=1}^{N(T)}(1+g(r_{t_{k+1}}))I(t_{k+1}=\tau_j).
\end{split}
\end{equation}
Notice that $y(t_{k+1})$ in (\ref{eq:glsn}) is not the exact value of $y(t)$
at the division points $t_{k+1}$, because $r(s)$ is not
necessarily constant on $[t_k,t_{k+1}]$. However, since
\[P\{r(t_{k+1})=i|r(t_k)=i\}=1+q_{ii}(t_{k+1}-t_k)+o(t_{k+1}-t_k)\rightarrow1,\]
as $\Delta\downarrow0$, for sufficiently small $\Delta$, it is
reasonable to use (\ref{eq:glsn}) as an approximation of the exact solution
of $y(t)$.

For the discrete approximate solution (\ref{eq:gln}) at the division points
$t_{k+1}$, we have
\begin{equation}
\begin{split}
X_{t_{k+1}}&=X_{t_k}+X_{t_k}\mu(r_{t_k})(t_{k+1}-t_{k})+X_{t_k}\sigma(r_{t_k})(W({t_{k+1}})-W({t_k}))\\
&\quad +\sum_{j=1}^{N(T)}X_{t_k}g(r_{t_k})I_{\{t_{k+1}=\tau_j\}}.
\end{split}
\end{equation}

Now we take specific data to computer the above example. Let
$r(t)$ be a right-continuous Markov chain taking values in
$S=\{1,2\}$ with generator
\[Q=\begin{bmatrix}-0.5 & 0.5 \\ 0.5 & -0.5 \end{bmatrix},\]
therefore, the one step transition probability from $r(t_k)$ to
$r(t_{k+1})$ is $e^{(t_{k+1}-t_k)Q}$.

Let $\mu(r(t))$, $\sigma(r(t))$, $g(r(t))$ take values as follows
\begin{eqnarray*}\left\{
\begin{array}{l}
\displaystyle \mu(1)=0.15,\quad \sigma(1)=0.1,\quad g(1)=-0.2;\\[.5cm]
\displaystyle \mu(2)=0.05,\quad \sigma(2)=0.1,\quad g(2)=-0.1.\\
\end{array}
\right.
\end{eqnarray*}

Let the Poisson intensity $\lambda=1$, and choose initial values
$y_0=10$, $r_0=1$, $T$ is fixed at 10. By applying the previously
described procedure, in Figure 1 the trajectory of the approximate
solution $X(t)$ with $\Delta=0.01$ is constructed.

\includegraphics[height=2.6in,width=4.1in]{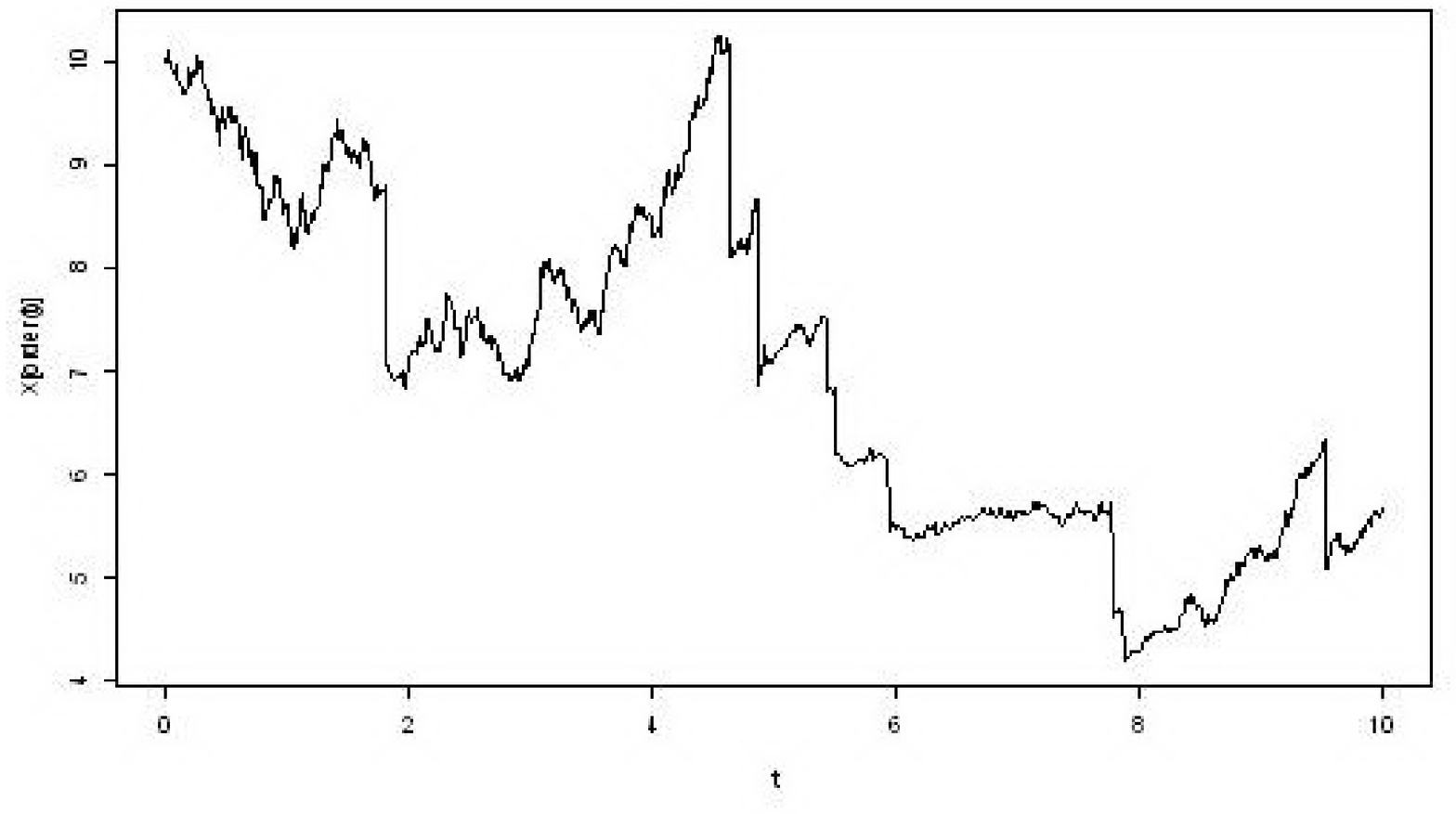}
\begin{center}\scriptsize{\textbf{Figure 1}. The trajectory of the approximate solution when $\Delta=0.01$\quad\quad\quad\quad}\\
\end{center}

To carry out the numerical simulation we successively choose the
stepsize $\Delta$ as the following Table 1, and for each $\Delta$,
we repeatedly simulate and compute
$\displaystyle\sup_{t_k\in[0,10]}(|X(t_k)-y(t_k)|^2)$ for 1000
times, then calculate the sample mean
$\widehat{E}(\displaystyle\sup_{t_k\in [0,10]}|X(t_k)-y(t_k)|^2)$.
The results are listed in the following Table 1.

\begin{tabular}{|c|c|}\hline
$\quad\quad\Delta\quad\quad$ & $\quad\quad\quad\quad\quad\quad\widehat{E}(\displaystyle\sup_{t_k\in[0,10]}|X(t_k)-y(t_k)|^2)\quad\quad\quad\quad\quad\quad$ \\
\hline 0.001 & 0.000179654\\
\hline 0.005 & 0.000908019\\
\hline 0.01 & 0.001793076\\
\hline 0.02 & 0.003626408\\
\hline 0.03 & 0.005427134\\
\hline 0.05 & 0.009155912\\
\hline 0.08 & 0.014817250\\
\hline 0.1 & 0.018204170\\
\hline
\end{tabular}
\begin{center}\scriptsize{\textbf{Table 1}. Estimation of the error between
the numerical and exact solutions\quad\quad\quad\quad}\\
\end{center}

Clearly, the numerical method reveals that the numerical solution
$X(t)$ converges to the exact solution $y(t)$ in $L^2$ as step
size $\Delta\downarrow0$, and the order of convergence is
one-half, i.e.,
\begin{equation*}
E(\sup_{0\leq t\leq T}|X(t)-y(t)|^2)\leq C\Delta+o(\Delta),
\end{equation*}
which strongly demonstrate
our result.\\

\noindent \textsl{Example 4.2}.  In this example, as an
application, we turn our attention to the expected ruin time for a
surplus process with regime switching. Denote by $r(t)$ the
external environment process, which influences the frequency of
claims and the distribution of claims. Suppose that the process
$r(t)$ is a homogeneous, irreducible and recurrent Markov process
with finite state space $S=\{1,2,\cdots,N\}$ with generator
$Q=(q_{ij})_{N\times N}$ and stationary distribution
$\pi=(\pi_1,\pi_2,\cdots,\pi_N)$. $N(t)$ is a Markov-modulated
Poisson process which accounts the arrival of claims, i.e. at time
$t$, claim arrival occur accord to a Poisson process with constant
intensity $\lambda_i$ when $r(t)=i$, and the corresponding claim
size $U_j$ have distribution $F_i(x)$ with means $\mu_i$. The
surplus process is given by
\begin{equation}\label{eq:sp}
Y(t)=u+t-\sum_{n=1}^{N(t)}U_n,
\end{equation}
where the aggregate premium received with rate 1 during interval
$(0,t]$ and $u$ is the initial reserve, the i.i.d. random
variables $U_n$ and $N(t)$ are conditionally independent given
$r(t)$.

Define $T^*=inf\{t>0\mid Y(t)<0\}$ to be the time of ruin of
(\ref{eq:sp}), and define the expected ruin time, given that the initial
environment state is $i$ and the initial reserve is $u$, by
$\xi_{i}(u)=E[T^*\mid r(0)=i,Y(0)=u]$. Let
$\eta=\sum_{i=1}^{N}\pi_i\lambda_i\mu_i$ and $\rho=\eta^{-1}-1$.
It's known that $\rho<0$ ensures $\xi_{i}(u)<\infty$.

Now we consider the special case of two environmental states. Let
\begin{equation*}Q=
\begin{pmatrix}
-q_1 & q_1 \\
q_2 & -q_2
\end{pmatrix},
\end{equation*}
then the stationary distribution of $r(t)$ is $\displaystyle
\pi_1=\frac{q_2}{q_1+q_2}$,
$\displaystyle\pi_2=\frac{q_1}{q_1+q_2}$。 Suppose that claim
sizes $U_1,U_2,\cdots$ are independent and exponentially
distributed with mean $\mu$, which are independent of the
environment process for convenience. Thus according to \cite{24},
we have
\begin{Pro}
If $\rho<0$, then it holds
\begin{equation}\label{eq:two}
\begin{split}
&\xi_{1}(u)=A_1+\frac{u}{\eta-1}+Be^{ku},\\
&\xi_{2}(u)=A_2+\frac{u}{\eta-1}+BD(k)e^{ku},
\end{split}
\end{equation}
where $k$ is the unique negative root of the following equation
\begin{equation}\label{eq:nre}
\begin{split}
P(k)&=k^3+k^2(\rho_1+\rho_2-q_1-q_2)\\
&\quad+k(\rho_1\rho_2-\rho_1q_2-\rho_2q_1-\frac{q_1}{\mu}-\frac{q_2}{\mu})-\frac{\rho_1q_2}{\mu}-\frac{\rho_2q_1}{\mu}\\
&=0,
\end{split}
\end{equation}
with $\displaystyle \rho_i=\frac{1}{\mu}-\lambda_i,\ i=1,2$, and
$(A_1, B, A_2)$ is the unique solution of the following liner
equation
\begin{equation}\label{eq:le}
\begin{pmatrix}
q_1 & 0 & -q_1 \\
1 & 1/(k\mu+1) & 0\\
0 & D(k)/(k\mu+1) & 1
\end{pmatrix}
\begin{pmatrix}
A_1\\
B\\
A_2
\end{pmatrix}
=\frac{1}{\eta-1}
\begin{pmatrix}
\eta-\lambda_1\mu\\
\mu\\
\mu
\end{pmatrix},
\end{equation}
where
\begin{equation}
D(k)=\frac{q_1+k(\mu q_1+\mu\lambda_1-1)-k^2\mu}{q_1+k\mu q_1}.
\end{equation}
\end{Pro}

From Proposition 4.1, it is easily to calculate $\xi_{1}(u)$ and
$\xi_{2}(u)$ if the appropriate values of parameters in the model
are given. However, in this paper we will use the numerical scheme
described in Section 2 to generate the simulating paths and then
derive the simulated value of $\xi_{1}(u)$ and $\xi_{2}(u)$.
Generally speaking, this simulation approach is feasible and
efficient for obtaining numerical solutions to problems which are
too complicated to solve analytically. Here we take for calculating
$\xi_{1}(u)$ as an example.

Let
\begin{equation}
Q=
\begin{pmatrix}
-1 & 1 \\
1 & -1
\end{pmatrix},
\end{equation}
and $\lambda_1=1$, $\lambda_2=2$, $\mu=1$, then it is easy to
obtain that $\displaystyle\pi=(\frac{1}{2},\frac{1}{2})$,
$\displaystyle\eta=\frac{3}{2}$,
$\displaystyle\rho=-\frac{1}{3}<0$, $\rho_1=0$, $\rho_2=-1$.
Solving equations (\ref{eq:nre}) and (\ref{eq:le}) respectively, we get
$k=-0.6751309$, $A_1=2.712742$, $B=1.481194$. Substituting into
(4.7) gives the exact value of $\xi_{1}(u)$
\begin{equation}\label{eq:ev}
\xi_{1}(u)=2.712742+2u+1.481194e^{-0.6751309u}.
\end{equation}

For comparing the exact value (\ref{eq:ev}) with the simulated value of
$\xi_{1}(u)$, we adopt the numerical algorithm in Section 2. After
simulating the trajectory of $Y(t)$, which is also denoted by
$X(t)$, it is easy to explore the trajectory behavior on $[0,T]$.
To avoid unnecessary error, we choose a sufficient large interval
$[0,T]$ and record the ruin time $T^*$ of $X(t)$ for every
trajectory. Set $T^*=T$ if $T^*>T$. In the following, let the
stepsize $\Delta=0.01$ and $T=100$, for each initial reserve
$u=5,8,10,15,20$, we repeatedly simulate and compute $T^*$ for
1000 times respectively, then calculate the sample mean
$\bar{\xi_{1}}(u)$. The results are listed in the following Table 2.\\

\begin{tabular}{|c|c|c|}\hline
$\quad\quad\quad u \quad\quad\quad$ & $\quad\quad\quad\quad \xi_{1}(u)\quad\quad\quad\quad$ & $\quad\quad\quad\quad\bar{\xi_{1}}(u)\quad\quad\quad\quad$\\
\hline 5 & 12.76339 & 12.4102\\
\hline 8 & 18.71942 & 18.7526\\
\hline 10 & 22.71447 & 23.1208\\
\hline 15 & 32.71280 & 32.5327\\
\hline 20 & 42.71274 & 41.9659\\
\hline
\end{tabular}
\begin{center}\scriptsize{\textbf{Table 2}. Exact and simulated value of the expected ruin time for initial reserve $u$}\\
\end{center}

Clearly, the simulated value $\bar{\xi_{1}}(u)$ is considerably
close to $\xi_{1}(u)$, which again demonstrate the efficiency of
our numerical method.

\end{document}